\documentclass[letter, 11pt]{article}

\usepackage
{
makeidx,epsfig, amstext, setspace, amsmath, amssymb,amsthm,amsfonts,times,latexsym,mathptmx,algorithmic,
wrapfig,bbm,graphicx,color,enumerate,endnotes
}
\usepackage{float}
\usepackage[ruled,linesnumbered]{algorithm2e}
\usepackage[top=2.5cm, bottom= 2.5cm, left= 2.5cm, right= 2.5cm]{geometry}
\newtheorem{theorem}{Theorem}
\newtheorem{CO}[theorem]{Corollary}
\newtheorem{LE}[theorem]{Lemma}
\newtheorem{lemma}[theorem]{Lemma}

\newcounter{claim_nb}[theorem]
\newtheorem{claim}[claim_nb]{Claim}
\newtheorem*{claim*}{Claim}

\newcommand{\zB}{\mathcal B}
\newcommand{\zC}{\mathcal C}

\newcommand{\zR}{\mathcal R}
\newcommand{\zP}{\mathcal P}
\newcommand{\zQ}{\mathcal Q}

\DeclareMathOperator{\im}{im}

\newcommand{\Aone}{A1 }
\newcommand{\Atwo}{A2 }
\newcommand{\Athr}{A6 }
\newcommand{\Afou}{A3 }
\newcommand{\Afiv}{A4 }
\newcommand{\Asix}{A5 }
\newcommand{\Asev}{A7 }

\newcommand{\ignore}[1]{}

\newcommand{\obar}[1]{\mkern 1.5mu\overline{\mkern-1.5mu#1\mkern-1.5mu}\mkern 1.5mu}

\title{Half-integral linkages in highly connected directed graphs\thanks{Supported by the European Research Council under the European Unions Seventh Framework Programme (FP7/2007-2013)/ERC Grant Agreement no. 279558}}

\author{
Katherine Edwards\thanks{\texttt{katherine.edwards2@gmail.com}}
\and
Irene Muzi\thanks{Department of Computer Science, University of Rome, ``La Sapienza", Rome, Italy.}
\and
Paul Wollan\thanks{Department of Computer Science, University of Rome, ``La Sapienza", Rome, Italy.  email: \texttt{wollan@di.uniroma1.it}.}
}

\begin{document}
\maketitle

%!TEX root = journal.tex
\begin{abstract}
We study the half-integral $k$-Directed Disjoint Paths Problem ($\tfrac12$kDDPP) in highly strongly connected digraphs.
The integral kDDPP is NP-complete even when restricted to instances where $k=2$, and the input graph is $L$-strongly connected, for any $L\geq 1$.
We show that when the integrality condition is relaxed to allow each vertex to be used in two paths, the problem becomes efficiently solvable in highly connected digraphs (even with $k$ as part of the input).
Specifically, we show that there is an absolute constant $c$ such that for each $k\geq 2$ there exists $L(k)$ such that $\tfrac12$kDDPP is solvable in time $O(|V(G)|^c)$ for a $L(k)$-strongly connected directed graph $G$.
As the function $L(k)$ grows rather quickly, we also show that $\tfrac12$kDDPP is solvable in time $O(|V(G)|^{f(k)})$ in $(36k^3+2k)$-strongly connected directed graphs.
We also show that for each $\epsilon<1$ deciding half-integral feasibility of kDDPP instances is NP-complete when $k$ is given as part of the input, even when restricted to graphs with strong connectivity $\epsilon k$.
\end{abstract}

%%%%%%%%%%%%%%%%%%%%%%%%%%%%%%%%%%%%%%%%%%%%%%
%%%%%%%%%%%%%%%%%%%%%%%%%%%%%%%%%%%%%%%%%%%%%%

%!TEX root = journal.tex

\section{Introduction}

Let $k \ge 1$ be a positive integer.  An \emph{instance of a directed $k$-linkage problem} is an ordered tuple $(G, S, T)$ where $G$ is a directed graph and $S = (s_1, \dots, s_k)$ and $T = (t_1, \dots, t_k)$ are each ordered sets of $k$ distinct vertices in $G$.  The instance is \emph{integrally feasible} if there exist paths $P_1, \dots, P_k$ such that $P_i$ is a directed path from $s_i$ to $t_i$ for $1 \le i \le k$ and the paths $P_i$ are pairwise vertex disjoint.  The paths $P_1, \dots, P_k$ will be referred to as an \emph{integral solution} to the linkage problem.  

The \emph{$k$-Directed Disjoint Paths Problem} ({\bf kDDPP}) takes as input an instance of a directed $k$-linkage problem.  If the problem is integrally feasible, we output an integral solution and otherwise, return that the problem is not feasible.
The kDDPP is notoriously difficult.  The problem was shown to be NP-complete even under the restriction that $k = 2$ by Fortune, Hopcroft and Wyllie \cite{FHW}.  

In an attempt to make the kDDPP more tractable, Thomassen asked if the problem would be easier if we assume the graph is highly connected.  Define a \emph{separation} in a directed graph $G$ as a pair $(A, B)$ with $A, B \subseteq V(G)$ such that $A \cup B = V(G)$ and where there does not exist an edge $(u,v)$ with $ u \in A \setminus B$ and $v \in B \setminus A$.  The \emph{order} of the separation $(A, B)$ is $|A \cap B|$.  The separation is \emph{trivial} if $A \subseteq B$ or $B \subseteq A$.  The graph $G$ is \emph{strongly $k$-connected} if $|V(G)| \ge k+1$ and there does not exist a nontrivial separation of order at most $k-1$.  Let $k \ge 1$ and define a directed graph $G$ to be \emph{integrally $k$-linked} if every linkage problem $(G, S, T)$ is integrally feasible.  Thomassen conjectured \cite{Thomassen1} that there exists a function $f$ such that every $f(k)$-strongly connected digraph $G$ is integrally $k$-linked.  He later answered his own conjecture in the negative \cite{Thomassen2}, showing that no such function $f(k)$ exists.  Moreover, he also showed \cite{Thomassen2} for all $L \ge 1$, the 2DDPP is NP-complete even when restricted to problem instances where the graph is $L$-strongly connected.  

In this article, we relax the kDDPP problem by requiring that a potential solution not use any vertex more than twice. Define a directed $k$-linkage problem $(G,S, T)$ to be \emph{half-integrally feasible} if $S = (s_1, \dots, s_k)$, and $T = (t_1, \dots, t_k)$ and there exist paths $P_1, \dots, P_k$ such that:
\begin{itemize}
\item for all $1 \le i \le k$, $P_i$ is a directed path from $s_i$ to $t_i$, and
\item for every vertex $v \in V(G)$, $v$ is contained in at most two distinct paths $P_i$.
\end{itemize}
The paths $P_1, \dots, P_k$ form a \emph{half-integral solution}.  

The main result of this article is that the $\frac{1}{2}$kDDPP is polynomial time solvable (even with $k$ as part of the input) when the graph is sufficiently highly connected.  Define a graph $G$ to be \emph{half-integrally $k$-linked} if every $k$ disjoint paths problem $(G, S,T)$ is half-integrally feasible.  
\begin{theorem}\label{thm:linked}
For all integers $k \ge 1$, there exists a value $L(k)$ such that every strongly $L(k)$-connected graph is half-integrally $k$-linked.  Moreover, there exists an absolute constant $c$ such that given an instance $(G, S, T)$ of the $\frac{1}{2}$kDDPP where $G$ is $L(k)$-connected, we can find a solution in time $O(|V(G)|^c)$.  
\end{theorem}
The assumption that $G$ is highly connected in Theorem \ref{thm:linked} cannot be omitted under the usual complexity assumptions.

\begin{theorem}\label{thm:hardness}
For all $\epsilon < 1$, it is NP-complete to determine whether a given kDDPP instance $(G, S, T)$ half-integrally feasible, even under the assumption that $G$ is $\epsilon k$-strongly connected.
\end{theorem}  

The value for $L(k)$ in Theorem \ref{thm:linked} grows extremely quickly.  However, when we fix $k$, we can still efficiently solve the $\frac{1}{2}$kDDPP with a significantly weaker bound on the connectivity than that given in Theorem \ref{thm:linked}.

\begin{theorem}\label{thm:2kconn}
There exists a function $f$ satisfying the following.  Let $k \ge 1$ be a positive integer.   Given a $k$-linkage problem $(G, S, T)$ such that $G$ is $(36k^3 + 2k)$-strongly connected, we can determine if the problem is half-integrally feasible and if so, output a half-integral solution, in time $O(|V(G)|^{f(k)})$. 
\end{theorem}

Given that the kDDPP is NP-complete even in the case $k=2$, previous work on the problem has focused on various relaxations of the problem.  Schrijver \cite{Schrijver94} showed that for fixed $k$, the kDDPP is polynomial time solvable when the input graph is assumed to be planar.  Later, Cygan et al \cite{Cygan2013} improved this result, showing that the kDDPP is fixed parameter tractable with the assumption that the input graph is planar.  In their recent series of articles \cite{KK1, KKK, digrid} leading to the breakthrough showing the grid theorem holds for directed graphs, Kawarabayashi and Kreutzer and Kawarabayashi et al showed the following relaxation of the kDDPP can be efficiently resolved for fixed $k$.  They showed that there exists a polynomial algorithm which, given an instance $(G, S = (s_1, \dots, s_k), T = (t_1, \dots, t_k))$ of the kDDPP, does one of the following:
\begin{itemize}
\item find directed paths $P_i$, $1 \le i \le k$, such that $P_i$ links $s_i$ to $t_i$ and for every vertex $v$ of $G$, $v$ is in at most four distinct $P_i$, or
\item determine that no integral solution to $(G, S, T)$ exists.
\end{itemize}
In terms of hardness results, Slivkins \cite{Slivkins} showed that the kDDPP is $W[1]$-complete even when restricted to acyclic graphs.  Kawarabayashi et al \cite{KKK} announced that the proof of Slivkins result can be extended to show that the $\frac12$kDDPP is also $W[1]$-complete.

There are two primary steps in the proof of Theorem \ref{thm:linked}.  First, we show that any highly connected graph contains a large structure which we can use to connect up the appropriate pairs of vertices.   The exact structure we use is a \emph{bramble of depth two}.  A bramble is a set of pairwise touching, connected (strongly connected) subgraphs; they are widely studied certificates of large tree-width both in directed and undirected graphs.  See Sections \ref{sec:tw} and \ref{sec:brambles} for the exact definitions and further details.  The existence of such a bramble of depth two follows immediately from Kawarabayashi and Kreutzer's proof of the grid theorem \cite{KKfull}; however, the algorithm given in \cite{KKfull} only runs in polynomial time for fixed size of the bramble.  We show in Section \ref{sec:findingbramble} that from appropriate assumptions which will hold both in the proof of Theorem \ref{thm:linked} and Theorem \ref{thm:2kconn}, we are able to find a large bramble of depth two in time $O(n^c)$ for a graph on $n$ vertices and some absolute constant $c$.  

The second main step in the proof of Theorem \ref{thm:linked} is to show how we can use such a bramble of depth two to find the desired solution to a given instance of the $\frac12$kDDPP.  Define a \emph{linkage} to be a set of pairwise disjoint paths.  We show in Section \ref{sec:linking} that given an instance $(G, S, T)$ and a large bramble $\zB$ of depth two, we can find a smaller, sub-bramble $\zB' \subseteq \zB$ along with a linkage $\zP$ of order $k$ such that every element of $\zP$ is a path from an element of $S$ to a distinct subgraph in $\zB'$.  Moreover, the linkage $\zP$ is internally disjoint from $\zB'$.  At the same time, we find a linkage $\zQ$ from distinct subgraphs of $\zB'$ to the vertices $T$.  Thus, by linking the appropriate endpoints of $\zQ$ and $\zP$ in the bramble $\zB'$, we are able to find the desired solution to $(G, S, T)$.  The fact that the bramble $\zB'$ has depth two ensures that the solution we find uses each vertex at most twice.  
This result is given as Theorem \ref{thm:main}; the statement and proof are presented in Section \ref{sec:linking}.  

Linking to a well-behaved structure (the bramble of depth two in the instance above) is a common technique in disjoint path and cycle problems in undirected graphs.  See \cite{Kawarabayashi2008, GMXIII} for examples.  The main contribution of Theorem \ref{thm:main} is to extend the technique to directed graphs, and in particular, simultaneously find the linkage from $S$ to $\zB'$ and the linkage $\zQ$ from $\zB'$ to $T$.  This is made significantly more difficult in the directed case by the directional nature of separations in directed graphs and the fact that there is no easy way to control how the separations between $S$ and $\zB'$ and those between $\zB'$ and $T$ cross.  

The proofs of Theorems \ref{thm:linked} and \ref{thm:2kconn} are given in Section \ref{sec:proofs} and the construction showing NP-completeness in Theorem \ref{thm:hardness} is given in Section \ref{sec:lowerbounds}.

%!TEX root = journal.tex
\section{Directed tree-width}\label{sec:tw}

An \emph{arborescence} is a directed graph $R$ such that $R$ has a vertex $r_0$, called the \emph{root} of $R$, with the property that for every vertex $r \in V (R)$ there is a unique directed path from $r_0$ to $r$. Thus every arborescence arises from a tree by selecting a root and directing all edges away from the root. If $r, r' \in V (R)$ we write $r' > r$  if $r' \neq r$ and there exists a directed path in $R$ from $r$ to $r'$.  If $(u,v) \in E(R)$ and $r \in V(R)$, we write $r > (u,v)$ if $r > v$ or $r = v$.  Let $G$ be a directed graph and $Z \subseteq V(G)$. A set $S \subseteq V(G) \setminus Z$ is \emph{$Z$-normal} if there is no directed walk in $G - Z$ with the first and last vertex in $S$ which also contains a vertex of $V(G) \setminus (S \cup Z)$.  Note that every $Z$-normal set is a union of strongly connected components of $G-Z$.  

Let $G$ be a directed graph.  A \emph{tree decomposition} of $G$ is a triple $(R,\beta,\gamma)$, where $R$ is an arborescence, $\beta: V(R) \rightarrow 2^{V(G)}$ and $\gamma: E(R) \rightarrow 2^{V(G)}$ are functions such that:
\begin{enumerate}
\item $\{\beta(r): r \in V(R)\}$ is a partition of $V(G)$ into non-empty sets and 
\item if $e \in E(R)$, then $\{\beta(r): r \in V(R), r > e\}$ is $\gamma(e)$-normal.
\end{enumerate}
The sets $\beta(r)$ are called the \emph{bags} of the decomposition and the sets $\gamma(e)$ are called the \emph{guards} of the decomposition.  For any $r \in V(R)$, we define $\Gamma(r) := \beta(r) \cup \{\gamma(e): \text{ $e$ incident to $r$}\}$. The \emph{width} of $(R, \beta, \gamma)$ is the smallest integer $w$ such that $|\Gamma(r)| \le w+1$ for all $r \in V(R)$.  The \emph{directed tree-width} of $G$ is the minimum width of a tree decomposition of $G$.

Johnson, Robertson, Seymour, and Thomas showed that if we assume $k$ and $w$ are fixed positive integers, then we can efficiently resolve the kDDPP when restricted to directed graphs of tree-width at most $w$ \cite{JRST}.
\begin{theorem}[\cite{JRST}, Theorem 4.8]\label{thm:DP}
For all $t \ge 1$, there exists a function $f$ satisfying the following.  Let $k \ge 1$, and let $(G, S, T)$ be an $k$-linkage problem such that the directed tree-width of $G$ is at most $t$.  Then we determine if $(G, S, T)$ is integrally feasible and if so, output an integral solution, in time $O(|V(G)|^{f(k)})$.
\end{theorem}

A simple construction shows that the same result holds to efficiently resolve $k$-linkage problems half-integrally when $k$ and the tree-width of the graph are fixed.  We first define the following operation.  To \emph{double} a vertex $v$ in a directed graph $G$, we create a new vertex $v'$ and add the edges $(u, v')$ for all edges $(u, v) \in E(G)$, the edges $(v', u)$ for all edges $(v,u)\in E(G)$ and the edges $(v,v')$ and $(v',v)$.

\begin{CO}\label{cor:solveboundedtw}
For all $t \ge 1$, there exists a function $f$ satisfying the following.  Let $k \ge 1$, and let $(G, S, T)$ be an instance of a $k$-linkage problem such that the directed tree-width of $G$ is at most $t$.  Given in input $(G, S, T)$ and a directed tree-decomposition of $G$ of width at most $t$, we can determine if the problem is half-integrally feasible and if so, output a half-integral solution, in time $O(|V(G)|^{f(k)})$.  
\end{CO}

\begin{proof}
Fix $w \ge 1$ to be a positive integer.  Let $(G, S=(s_1, \dots, s_k), T=(t_1, \dots, t_k))$ be an instance of a $k$-linkage problem where $G$ has tree-width at most $w$.  Let $G'$ be the directed graph obtained by doubling every vertex $v \in V(G)$.  Define the $k$-linkage problem $(G', S^*=(s_1^*, \dots, s_k^*), T^*=(t_1^*, \dots, t_k^*))$ by letting $s_i^* = s_i$ and $t_i^* = t_i'$ for $1 \le i \le k$.  Thus, $(G, S, T)$ is half-integrally feasible if and only if $(G', S^*, T^*)$ is integrally feasible.  Moreover, any integral solution to $(G', S^*, T^*)$ can be easily converted to a half-integral solution for the original problem $(G, S, T)$.

Let $(R, \beta, \gamma)$ be a tree decomposition of $G$ of width $w$.  Observe that $(R, \beta', \gamma')$ defined by $\beta'(r) = \{\{v, v'\}: v \in \beta(r)\}$ and $\gamma'(r) = \{\{v, v'\}: v \in \gamma(r)\}$ yields a tree decomposition of $G'$ of width at most $2w$.  Thus, by Theorem \ref{thm:DP}, we can determine if $(G', S^*=(s_1^*, \dots, s_k^*), T^*=(t_1^*, \dots, t_k^*))$  is integrally feasible and find an solution when it is, in polynomial time assuming $k$ and $w$ are fixed, proving the claim.
\end{proof}
%!TEX root = journal.tex
\section{Certificates for large directed tree-width}\label{sec:brambles}

A \emph{bramble} in a directed graph  $G$ is a set $\zB$ of strongly connected subgraphs $B \subseteq G$ such that if $B,B' \in \zB$, then $V(B) \cap V(B') \neq \emptyset$ or there exists edges $e, e' \in E(G)$ such that $e$ links $B$ to $B'$ and $e'$ links $B'$ to $B$. A cover of $\zB$ is a set $X \subseteq V(G)$ such that $V(B) \cap X \neq \emptyset$ for all $B \in \zB$.  The \emph{order} of a bramble is the minimum size of a cover of $\zB$.  The \emph{bramble number}, denoted $bn(G)$, is the maximum order of a bramble in $G$.
The elements of a bramble are called \emph{bags}, and the \emph{size} of a bramble, denoted $|\zB|$, is the number of bags it contains.

The bramble number of a directed graph gives a good approximation of the tree-width, as seen by the following theorem of \cite{Reed} as formulated by \cite{digrid}.

\begin{theorem}[\cite{Reed},\cite{digrid}]
There exist constants $c$, $c'$ such that for all directed graphs $G$, it holds that
$$ bn(G) \le c \cdot tw(G) \le c'\cdot bn(G).$$
\end{theorem}

Johnson, Robertson, Seymour, and Thomas showed one can efficiently either find a large bramble in a directed graph or explicitly find a directed tree-decomposition.  Note that the result is not stated algorithmically, but that the algorithm follows from the construction in the proof.  Additionally, they looked at an alternate certificate of large tree-width, namely havens, but a haven of order $2t$ immediately gives a bramble of order $t$ by the definitions.
\begin{theorem}[\cite{JRST}, 3.3]\label{thm:brambleortd}
There exist constants $c_1, c_2$ such that for all $t$ and directed graphs $G$, we can algorithmically find in time $|V(G)|^{c_1}$ either a bramble in $G$ of order $t$ or a tree-decomposition of $G$ of order at most $c_2t$.  Moreover, if we find the bramble, it has at most $|V(G)|^{2t}$ elements.
\end{theorem}

A long open question of Johnson, Robertson, Seymour, and Thomas \cite{JRST} was whether sufficiently large tree-width in a directed graph would force the presence of a large directed grid minor.  Let $r \ge 2$ be a positive integer.  The \emph{directed $r$-grid $J_r$} is the graph defined as follows.  Let $C_1, \dots, C_r$ be directed cycles of length $2r$.  Let the vertices of $C_i$ be labeled $v_1^i, \dots, v_{2r}^i$ for $1 \le i \le r$.  For $1 \le i \le 2r$, $i$ odd, let $P_i$ be the directed path $v_i^1, v_i^2, \dots, v_i^r$.  For $1 \le i \le 2r$, $i$ even, let $P_i$ be the directed path $v_i^r, v_i^{r-1}, \dots, v_i^1$.  The directed grid $J_r = \bigcup_1^r C_i \cup \bigcup_1^{2r} P_i$.  

In a major recent breakthrough, Kreutzer and Kawarabayashi have confirmed the conjecture of Johnson et al. 
\begin{theorem}[\cite{digrid}]
There is a function $f : \mathbb{N} \rightarrow \mathbb{N}$ such that given any directed graph and any fixed constant $k$, in polynomial time, we can obtain either
\begin{enumerate}
\item  a cylindrical grid of order $k$ as a butterfly minor, or
\item a directed tree decomposition of width at most $f(k)$.
\end{enumerate}
\end{theorem}

For our purposes, we will use brambles when attempting to solve the $\frac{1}{2}$kDDPP.  However, in order to ensure that the paths we find don't use any vertex more than twice, we require the bramble to have \emph{depth} two.  Define the \emph{depth} of a bramble $\zB = \{B_1, \dots, B_t\}$ in a directed graph $G$ to be the $\max_{v\in V(G)} |\{i: v \in V(B_i)\}|$; in other words, a bramble has depth at most $k$ for some positive integer $k$ if no vertex is contained in more than $k$ distinct subgraphs in the bramble.  Note that if $\zB$ has depth $k$ and size $t$, then it has order at least $\lceil t/k\rceil$.

\begin{LE}
For all $t \ge 2$, the directed $t$-grid contains a model of a bramble $\zB$ of size $t$ and depth two.
\end{LE}

\begin{proof}
Let the cycles $C_1, \dots, C_{k}$, paths $P_1, \dots, P_{2k}$, and vertex labels $v_i^j$, $1 \le i \le 2k$, $1 \le j \le k$, be as in the definition of the cylindrical grid.  For every $l$, $1 \le l \le k$, and for every $i$, $1 \le i \le 2k$, let $P_i(l)$ be the subpath of $P_i$ with endpoints $v_i^1$ and $v_i^l$.  For $1 \le i \le k-1$, let $C_i'$ be the (unique) cycle in $C_i \cup C_1 \cup P_{2i-1}(i) \cup P_{2i}(i)$ which contains all the vertices $v_1^j$, $1 \le j \le 2k$.  Let $C_k' = C_k$.  The cycles $C_1', \dots, C_k'$ form a bramble of depth two and size $k$, as desired. 
\end{proof}
%!TEX root = journal.tex
\section{Finding a bramble of depth two}\label{sec:findingbramble}

In this section, we show that given what we call a sufficiently large \emph{well-linked} set of vertices in a directed graph, we are able to efficiently find a large bramble of depth two.  The argument in many ways follows Diestel et al's proof of Robertson and Seymour's grid theorem (see \cite{diestel1999highly} for the proof) for undirected graphs.  We begin with a collection of disjoint linkages and show that in each of the linkages we can find a sublinkage which are pairwise disjoint.  We will need two classic results in graph theory, namely Ramsey and Menger's theorems.

\begin{theorem}[Menger's theorem~\cite{menger1927allgemeinen}]\label{thm:menger}
Let $G$ be a directed graph and $S,T$ subsets of $V(G)$.
The maximum number of vertex-disjoint $S-T$ paths equals the minimum order of a separation which separates $S$ from $T$.
Moreover, there exists an algorithm to find a maximum set of vertex-disjoint $S-T$ paths and a minimum order separation in time $O(|V(G)|^2)$.
\end{theorem}

\begin{theorem}[Ramsey's theorem]\label{thm:ramsey}
Let $r, t$ be positive integers.  For every (improper) two coloring of the edges of the undirected clique $K_{2^{r+t}}$ by red and blue, there exists either a subgraph $K_r$ with every edge is colored red or a subgraph $K_t$ with every edge colored blue.  Moreover, the desired $K_r$ or $K_t$ subgraph can be found in time $O(k^c)$ for some absolute constant $c$.  
\end{theorem}

We first give two preparatory lemmas before presenting the main result of this section.

\begin{LE}\label{lem:uncrossing1}
Let $G$ be a digraph on $n$ vertices and let $t, k\ge 2$ be positive integers.  Let $T = (10k)2^{2(t+k)}$.  Let $X = \{x_1, x_2, \dots, x_T\}$ and $Y= \{y_1, y_2,\dots, y_T\}$ be two disjoint sets of vertices of order $T$.  Let $\zP = \{P_1, P_2, \dots, P_T\}$ be a linkage from $X$ to $Y$ and $\zQ = \{Q_1, Q_2, \dots, Q_T\}$ a linkage from $Y$ to $X$, each of size $T$.  Assume each $P_i$ has endpoints $x_i$ and $y_i$ and assume that there exists a permutation $\pi$ of $[T]$ such that $Q_i$ has endpoints $y_i$ and $x_{\pi(i)}$.  Then one of the following holds:
\begin{enumerate}
\item there exist $B_1, \dots, B_t$ subgraphs of $G$ forming a bramble of size $t$ and depth two, or
\item there exists a subset $J \subseteq [T]$ with $|J| = k$ such that the subgraphs $P_j \cup Q_j \cup P_{\pi(j)} \cup Q_{\pi(j)} \cup P_{\pi(\pi(j))}$ are pairwise disjoint for all $j \in J$.
\end{enumerate}
Moreover, given $G$, $\zP$, and $\zQ$ in input, we can find either $B_1, \dots, B_t$ or $J$ in time $O(n^c)$ for some absolute constant $c$.
\end{LE}

\begin{proof}
Define an auxiliary undirected bipartite graph $H$ with vertex set $x_1, x_2, \dots, x_T, y_1, y_2, \dots, y_T$ and edges $x_iy_i$ and $x_{\pi(i)}y_i$ for all $1 \le i \le T$.  Thus, $H$ is the union of two perfect matchings and each component is either a cycle or a single edge.  For any induced subgraph $C$ of $H$, let $\overrightarrow{C} = \bigcup_{i:x_iy_i \in E(C)}P_i \cup \bigcup_{i: x_{\pi(i)}y_i\in E(C)}Q_i$.  Note that if $C$ is a connected component of $H$, then $\overrightarrow{C}$ is a strongly connected subgraph of $G$.  

Assume, as a case, that there are at least $T' = 2^{k+t}$ distinct components $C_1, \dots, C_{T'}$.  By Theorem \ref{thm:ramsey}, there exists $I \subseteq [T']$ such that one of the following holds:
\begin{itemize}
\item $|I| = t$ and the subgraphs $\overrightarrow{C_i}$ pairwise intersect for $i \in I$, or
\item $|I| = k$ and the subgraphs $\overrightarrow{C_i}$ are pairwise disjoint for $i \in I$.
\end{itemize}
Moreover, we can find $I$ in time $O(n^c)$.

In the first case, we claim that $\{\overrightarrow{C_i}:i \in I\}$ form a bramble of depth two; clearly by construction, the sets $\{\overrightarrow{C_i}:i \in I\}$ form a bramble.  To see that it has depth two, observe that $\zQ$ and $\zP$ are both linkages and given the fact that for any vertex $v$ such that $v$ is contained in  $V(\overrightarrow{C_i})$, it must be the case that $v$ is contained in some $P \in \zP$ or $Q \in \zQ$ which is a subpath of $\overrightarrow{C_i}$.   As the set elements of $\zP$ and $\zQ$ forming $\overrightarrow{C_i}$ and $\overrightarrow{C_j}$ are disjoint for all $i \neq j$, at most two $\overrightarrow{C_i}$ and $\overrightarrow{C_j}$ can intersect in the vertex $v$.  Thus the bramble is of depth two.

In the second case, for every $i \in I$, we fix $j = j(i)$ such that $P_j$ is contained in $\overrightarrow{C_i}$.  It follows that $J = \{j(i): i \in I\}$ satisfies outcome 2 in the statement of the lemma.

We conclude that there are at most $2^{k+t}$ distinct components, and thus, there exists a component $C$ of size at least $(10k)\cdot 2^{k+t}$.  Let $P_1, \dots, P_{2^{k+t}}$ be disjoint paths in $C$, each of length at least $10k$.

Assume that there exists an index $i$, $1 \le i \le 2^{k+t}$ such that for all edges $e$ and $f$ in $E(P_i)$ at distance at least six, $\overrightarrow{e} \cap \overrightarrow{f} = \emptyset$.  Fix pairwise disjoint subpaths $P_1', \dots, P_k'$ of $P$, such that 
\begin{itemize}
\item for $1 \le i \le k$, the path $P_i'$ has length five and contains exactly three edges $e$ such that $\overrightarrow{e} \in \zP$;
\item for $1 \le i < i' \le k$, $P_i'$ and $P_{i'}'$ are at distance at least five.  
\end{itemize}
By our assumption on $P_i$, the fact that each $P_i'$ starts and ends with an edge corresponding to a path in $\zP$, and the fact that the elements of $\zP$ are pairwise disjoint, it follows that for all $1 \le i < i' \le k$, $\overrightarrow{P_i'} \cap \overrightarrow{P_{i'}'} = \emptyset$, as in the outcome 2 of the lemma.  

We conclude that for all $1 \le i \le 2^{k+t}$, the subgraph $P_i$ has two edges at distance at least six such that their corresponding paths in $\zP \cup \zQ$ intersect.  It follows that there exists a five edge path $P_i'$ of $P_i$ containing three edges $e$ such that $\overrightarrow{e} \in \zP$ and a strongly connected subgraph ${H_i}$ of $G$ containing $\overrightarrow{P_i'}$.  By the same Ramsey argument as above applied to the subgraphs $H_1, \dots, H_{2^{k+t}}$, we see that one of the desired outcomes holds for $G$.
\end{proof}

A subset $X \subseteq V(G)$ of vertices of a directed graph $G$ is \emph{well-linked} if for any pair of subsets $U_1, U_2 \subseteq X$ with $|U_1| = |U_2|$, there exists a directed $U_1$ to $U_2$ linkage of order $|U_1|$.
 
\begin{LE}\label{lem:uncross2}
Let $X$ be a well-linked set in a directed graph $G$ on $n$ vertices.  Let $k, t \ge 2$ be positive integers.  Let $X_1, X_2, Y_1, Y_2 \subseteq X$ be pairwise disjoint subsets of $X$.  Let $\zP$ be a linkage of order $2^k2^T$ from $X_1$ to $X_2$ where $T = (10t)2^{4t}$.  Let $\zR$ be a linkage of order $2^k2^T$ from $Y_1$ to $Y_2$.  Then one of the following holds:
\begin{enumerate}
\item there exist $\zP' \subset \zP$ and $\zR' \subset \zR$ with $|\zP'| = |\zR'| = k$ such that for all $P \in \zP'$ and $R \in \zR'$, it holds that $P \cap R = \emptyset$, or 
\item there exists $B_1, \dots, B_t$ forming a bramble of size $t$ and depth two.
\end{enumerate}
Moreover, we can find the linkages satisfying outcome 1 or the bramble in 2 in time $O(n^c)$ for some absolute constant $c$.
\end{LE}

\begin{proof}
By Theorem \ref{thm:ramsey}, we may assume that there exist $\zP' \subseteq \zP$ and $\zR' \subseteq \zR$, each of order $T$ such that every element of $\zP'$ intersects every element of $\zR'$.  By the definition of well-linked set, there exists a linkage $\zQ^X$ from $X_2 \cap V(\zP')$ to $X_1 \cap V(\zP')$ and similarly, a linkage $\zQ^Y$ from $Y_2 \cap V(\zR')$ to $Y_1 \cap V(\zR')$.  By Theorem \ref{thm:menger}, we can find the linkages $\zQ^X$ and $\zQ^Y$ in time $O(n^2)$.  Label the elements of $\zR'$ as $R_1, \dots, R_T$ and the elements of $\zP'$ as $P_1, \dots, P_T$.  Let $\pi_X$ and $\pi_Y$ be two permutations of $[T]$ such that we can label the elements of $\zQ^X$ and $\zQ^Y$ as $Q_1^X, \dots, Q_T^X$ and $Q_1^Y, \dots, Q_T^Y$ such that $Q_i^X$ has a common endpoint with $P_i$ in $X_2$ and a common endpoint with $P_{\pi_X(i)}$ in $X_1$.  Similarly, $Q_i^Y$ has a common endpoint with $R_i$ in $Y_2$ and a common endpoint with $R_{\pi_Y(i)}$ in $Y_1$.  

Apply Lemma \ref{lem:uncrossing1} to the linkages $P_1, \dots, P_T$ and $Q_1^X, \dots, Q_T^X$.  We may assume that we get outcome 2 in the lemma.  Without loss of generality, we may assume that $\pi_X(i) = i+1$ for $1 \le i \le 3k-1$ and that the subgraphs $P_{3j+1} \cup Q_{3j+1}^X \cup P_{3j+2} \cup Q_{3j+2}^X \cup P_{3j+3}$ are pairwise disjoint for $0 \le j \le k-1$.  Similarly, by applying Lemma \ref{lem:uncrossing1} to $R_1, \dots, R_T$ and $Q_1^Y, \dots, Q_T^Y$ that $\pi_Y(i) = i+1$ for $1 \le i \le 3k-1$ and that the subgraphs $R_{3j+1} \cup Q_{3j+1}^Y \cup R_{3j+2} \cup Q_{3j+2}^Y \cup R_{3j+3}$ are pairwise disjoint for $0 \le j \le k-1$.

Since for all $0 \le j \le k-1$, the paths $P_{3j+3}$ intersects $R_{3j+1}$ and $R_{3j+3}$ intersects $P_{3j+1}$, we conclude that the subgraph
$$P_{3j+1} \cup Q_{3j+1} \cup P_{3j+2} \cup Q_{3j+2} \cup P_{3j+3} \cup R_{3j+1} \cup Q_{3j+1}^Y \cup R_{3j+2} \cup Q_{3j+2}^Y \cup R_{3j+3}
$$
contains a strongly connected subgraph $B_j$ which contains both $P_{3j+2}$ and $R_{3j+2}$.  Since every $P_i$ intersects every $R_{i'}$, we have that $\{B_0, \dots, B_{k-1}\}$ forms a bramble of size $k$ and depth two, as required.
\end{proof}

We now show the main result of the section which is that given a sufficiently large well-linked set, we can efficiently find a large bramble of depth two.

\begin{theorem}\label{thm:findbramble}
There exists a function $f$ which satisfies the following.  Let $G$ be a directed graph on $n$ vertices and $t \ge 1$ a positive integer.  Let $P$ be a directed path and $X \subseteq V(P)$ a well-linked set with 
$|X| \ge f(t).$
Then $G$ contains a bramble $\zB = B_1, \dots, B_t$ of depth two.  Moreover, given $G$, $P$, and $X$ in input, we can find $\zB$ in time $O(n^c)$ for some absolute constant $c$.
\end{theorem}

\begin{proof}
For a given $k \ge 1$, let $f_t(k) = 2^k2^{ (10t)2^{4t}}$.  We will use the notation $f_t^l(k)$ for the function $f_t$ iterated $l$ times, beginning with input $k$, i.e.~$f_t(f_t(f_t(\dots(f_t(k))\dots)))$.  Let $T = 2^{2t}$.   Fix pairwise disjoint subsets $X_i^{in}, X_i^{out}$ for $1 \le i \le T^2$ in $X$ such that there exist subpaths $P_1, \dots, P_{T^2}$ of $P$ which satisfy the following:
\begin{itemize}
\item for all $1 \le i < j \le T^2$, $P_i \cap P_j = \emptyset$; 
\item for all $i$, $X_i^{in} \cup X_i^{out} \subseteq V(P_i)$ and traversing the directed path $P_i$, the vertices of $X_i^{in}$ occur before the vertices of $X_i^{out}$.  
\end{itemize}
Moreover, we pick $X_i^{in}, X_i^{out}$, $1 \le i \le T^2$ such that 
$$|X_i^{in}|, |X_i^{out}| \ge f_t^{T^8}(2).$$
By assuming that the function $f$ in the statement of the theorem satisfies $f(t) \ge T^4 f_t^{T^8}(2)$, we see that such $X_i^{in}, X_i^{out}$ exist.

For all $1 \le i \le T^2$, $1 \le j \le T^2$, $i \neq j$, fix a directed linkage $\zQ(i, j)$ from $X_i^{out}$ to $X_j^{in}$ of order $f_t^{T^8}(1)$.  By Theorem \ref{thm:menger}, we can find such linkages $\zQ(i, j)$ in time $O(T^4n^2)$.  Fix a enumeration of $A := \{((i, j),(i', j')): 1 \le i, i', j, j' \le T^2, i \neq j, i' \neq j'\}$ and let $\alpha = |A| = T^8 - 2T^6 + T^4$.  For $0 \le l \le \alpha$, we define linkages $\zQ(i, j)^l \subseteq \zQ(i, j)$ of order $f_t^{T^8-l}(1)$ as follows.  Let $\zQ(i, j)^0 = \zQ(i, j)$.  For $l = 1, \dots, \alpha$, let $((i, j), (i', j'))$ be the $l^{th}$ pair in $A$.  Apply Lemma \ref{lem:uncross2} to the linkages $\zQ(i, j)^{l-1}$ and $\zQ(i', j')^{l-1}$.  We may assume that there exist disjoint sublinkages of $\zQ(i, j)^{l-1}$ and $\zQ(i', j')^{l-1}$, each of size $f_t^{T^8 - l}(1)$.  Call them $\zQ(i, j)^{l}$ and $\zQ(i', j')^l$, respectively.  For $(i'', j'')$ distinct from $(i, j)$ and $(i', j')$, fix $\zQ(i'', j'')^l$ to be an arbitrary subset of $\zQ(i'', j'')^{l-1}$ of order $f_t^{T^8 - l}(1)$.  We fix $Q(i, j)$ to be an element of the linkage $\zQ(i, j)^\alpha$.  By construction, the paths $Q(i, j)$ and $Q(i', j')$ are disjoint if $(i, j)\neq (i', j')$.  

Define strongly connected subgraphs $C_i$ and $R_i$ for $1 \le i \le T$ as follows.  Fix $i$, $1 \le i \le T$.  The subgraph $C_i = \bigcup_{j = 0}^{T-1} Q(i+jT, i + (j+1)T)$ along with the subpath of $P_{i + jT}$ linking endpoints of $Q(i +(j-1)T, i+jT)$ and $Q(i + jT, i + (j+1)T)$ for $0 \le j \le T-1$ where the values are taken modulo $T^2$.  Similarly, we define $R_i = \bigcup_{j = 1}^{T-1} Q(iT + j, iT + j + 1) \cup Q(iT +T, iT + 1)$ along with the subpaths of $P_{iT + j}$ linking the endpoints of $Q(iT + j-1, iT + j)$ and $Q(iT + j, iT + j + 1)$ for $2 \le j \le T-1$ along with the analogous subpaths of $P_{iT + 1}$ and $P_{iT + T}$.  If we think of the paths $P_i$ laid out in a $T \times T$ grid, the subgraphs $C_j$ are the natural strongly connected graphs formed by following the $Q(i', j')$ through a column of the grid, and the $R_j$ are the strongly connected graphs formed by the rows.

Every vertex in one of the subgraphs $C_i$ is either in a path $Q(i', j')$ or $P_{l'}$; moreover, the subset of $P_{l'}$ contained in $C_i$ are disjoint from those contained in $C_{i''}$ for $i' \neq i''$.  Thus, it is possible that two $C_i$ and $C_{i'}$ intersect in a vertex, but not for three distinct $C_i$, $C_{i'}$ and $C_{i''}$ to all intersect in a common vertex.  

By Theorem \ref{thm:ramsey}, there exists $I \subseteq [T]$ of size $t$ such that for $i, i' \in I$, either $C_i$ intersect or they are pairwise disjoint.  By the above observation, if they are pairwise intersecting, then $\{C_i: i \in I\}$ forms a bramble of size $t$ and depth two.  Thus, we may assume that $C_i \cap C_{i'} = \emptyset$ for all distinct $i, i' \in I$.  Similarly, there exists $J \subseteq [T]$, $|J| = t$ such that for all distinct $j, j' \in J$, $R_j \cap R_{j'} = \emptyset$.  Without loss of generality, assume that $I = J = [t]$.  The set $\{C_i \cup R_i : i \in [t]\}$ then forms a bramble of size $t$ and depth two, completing the proof.
\end{proof}

%!TEX root = journal.tex
\section{Linking in a bramble of depth two}\label{sec:linking}
The main result of this section is the following which shows that if we have a sufficiently large bramble of depth two, we can use it to efficiently resolve a given instance of the $\frac12$kDDPP under a modest assumption on the connectivity of the graph.
\begin{theorem}\label{thm:main}
For all $k \ge 1$, there exists a positive integer $t$ such that if $G$ is a $(36k^3+2k)$-strongly connected directed graph, and $G$ contains a bramble $\zB$ of depth two and size $t$, then for every $k$-linkage problem instance $(G,S, T)$ is half-integrally feasible.  Moreover, given $(G, S, T)$ and the bags of $\zB$, we can find a solution in time $O(k^4n^2)$.
\end{theorem}

We begin with some notation.  Recall that the doubling of a vertex in a directed graph was defined in Section \ref{sec:tw}.  
To \emph{contract} a set of vertices $U$ inducing a strongly connected subgraph of $G$ is to delete $U$ and create a new vertex $v$, then add edges $(w,v)$ for all edges $(w,u)\in E(G)$ with $u\in U,w\notin U$ and edges $(v,w)$ for all edges $(v,u)\in E(G)$ with $u\in U, w\notin U$.

Let $\zB$ be a depth two bramble in a directed graph $G$ and $\zB_1\subseteq \zB$.
Define the graph $G(\zB_1;\zB)$ as follows:
First, let $G'$ be the graph obtained from $G$ by doubling every vertex belonging to two bags of $\zB$ and to at least one bag of $\zB_1$. 
For each such vertex $v$, denote its double by $v'$. 
Let $\zB'$ be the collection of $|\zB_1|$ subsets of $V(G')$ obtained from $\zB_1$ by replacing each vertex $v$ belonging to a bag of $\zB$ with $v'$ in exactly one of the bags it belongs to.
Thus, the elements of $\zB'$ are pairwise disjoint and each induces a strongly connected subgraph of $G'$, so $\zB'$ is a depth 1 bramble in $G'$.
Let $G(\zB_1;\zB)$ be the graph obtained from $G'$ by contracting each element of $\zB'$. 
Denote by $K_{\zB_1}$ the set of contracted vertices in $G(\zB_1;\zB)$;
note that the vertices of $K_{\zB_1}$ form a bidirected clique.
Observe that every double of a vertex of $G'$ gets contracted, so $V(G(\zB_1;\zB)) \setminus K_{\zB_1} \subseteq V(G)$.
For a vertex $v\in K_{\zB_1}$, we write $\im(v)$ for the bag of $\zB_1$ corresponding to the vertices contracted to $v$.
We stress that each $\im(v)$ is a bag of $\zB_1$; in particular $\im(v)\subseteq V(G)$.

Let $S,T$ be disjoint subsets of the vertices of a directed graph $G$.
A separation $(A,B)$ \emph{separates} $S$ from $T$ if $S\subseteq A$ and $T\subseteq B$.
The separation $(A,B)$ \emph{properly separates} $S$ from $T$ if $S\setminus B$ and $T\setminus A$ are both nonempty.
For a positive integer $\alpha$, we say $S$ is \emph{$\alpha$-connected} to $T$ if every separation separating $S$ from $T$ has order at least $\alpha$.

Let $G$ be a directed graph, and $\zB'\subseteq \zB$ be brambles of depth two.
Let $X\subseteq V(G(\zB';\zB))\setminus K_{\zB'}$.
We say a $X-K_{\zB'}$ or $K_{\zB'}-X$ linkage $P_1,\dots,P_{|X|}$ is \emph{$\zB$-minimal} if 
% the paths are pairwise vertex-disjoint and each joins a vertex of $X$ to (respectively from) a vertex of $K_{\zB'}$ and 
none of the paths contains internally a vertex in $K_{\zB'}$ or in $\im(v)$ for some $v\in K_{\zB}\setminus \cup_i P_i$.

We now give a quick outline of how the proof will proceed.  
Let us denote $S = (s_1, \dots, s_k)$ and $T = (t_1, \dots, t_k)$.
Our approach to proving half-integral feasability is in two steps.
We find three sets of paths, one set of $k$ paths linking $S$ to the bramble $\zB$, another set linking $\zB$ to $T$, and a third linking the appropriate ends of paths in the first two sets to each other inside of $\zB$.
To get the first two sets of paths, we take advantage of the high connectivity of the graph.
Linking half-integrally inside of the bramble is easy, and its structure allows us to link any pairs of vertices we like half-integrally.
We need the union of the three sets of paths to form a half-integral solution, so we will choose the first and second sets each to be (almost) vertex-disjoint, and to intersect the bramble $\zB$ in a very limited way.
The third set of paths will be half-integral and completely contained in $\zB$.

The underlying idea behind our approach to finding the first two sets of paths is to contract each bag of the bramble (after doubling vertices in two bags) and try to apply Menger's theorem.
In trying to do this, some issues arise.
First, we want the ends of all $2k$ paths to belong to distinct bags of $\zB$.
More concerningly, contracting the bags of the bramble may destroy the connectivity between the bramble and the terminals $S$ and $T$.
We solve this by throwing away a bounded number of bags from the bramble until we are left with a sub-bramble that is highly connected to $S$ and from $T$.
In Subsection \ref{sec:linkup}, we will show how to find the first two sets of paths (Lemma \ref{lem:linkup}), modulo finding the sub-bramble (Lemma \ref{lem:linkcontracted}), and the third set of paths (Lemma \ref{lem:insideclique}).
Then we show how to put these pieces together to prove Theorem \ref{thm:main}.
In Subsection \ref{sec:littlebramble} we prove Lemma \ref{lem:linkcontracted}.

\subsection{Linking into and inside of a depth two bramble}\label{sec:linkup}

\begin{lemma}\label{lem:linkup}
Let $G$ be a $(36k^3 + 2k)$-strongly connected directed graph and $\zB$ be a bramble of depth two and size $>188k^3$ in $G$.
Let $(G, S = (s_1, \dots, s_k), T = (t_1, \dots, t_k))$ be a $k$-linkage problem instance.
Then we can find paths $P_1^s, \dots, P_k^s, P_1^t, \dots, P_k^t$ and $\zB' \subseteq \zB$ satisfying the following:
\begin{itemize}
\item[{\bf \Aone:}] For each $i$, $P_i^s$ is a directed path from $s_i$ to some vertex $s_i'$, and $P_i^t$ is a directed path from some vertex $t_i'$ to $t_i$.
\item[{\bf \Atwo:}] The vertices $s_1', \dots, s_k', t_1', \dots, t_k'$ belong to distinct bags of $\zB'$, say $B_1^s, \dots, B_k^s, B_1^t, \dots, B_k^t$, respectively.
\item[{\bf \Afou:}] Every vertex belongs to at most two of $P_1^s, \dots, P_k^s$, and if a vertex $v$ does belong to two paths, say $P_i^s$ and $P_j^s (i\neq j)$, then $v= s_i'$ or $v= s_j'$.
\item[{\bf \Afiv:}] Similarly, every vertex belongs to at most two of $P_1^t, \dots, P_k^t$, and if a vertex $v$ does belong to two paths, say $P_i^t$ and $P_j^t (i\neq j)$, then $v= t_i'$ or $v= t_j'$.
\item[{\bf \Asix:}] For each $i$, the internal vertices of $P_i^s$ and of $P_i^t$ belong to at most one bag of $\zB'$.
\item[{\bf \Athr:}] For each $i,j,\ell$ all distinct ,  $P_i^s \cap P_j^t \cap (B_{\ell}^s \cup B_{\ell}^t) = \emptyset$.
\item[{\bf \Asev:}] Every vertex belongs to at most two of $P_1^s, \dots, P_k^s, P_1^t, \dots, P_k^t$.

\end{itemize}
Moreover, given the bags of $\zB$, we can find the paths $P_1^s, \dots, P_k^s, P_1^t, \dots, P_k^t$ in time $O(k^4n^2)$.
\end{lemma}

We will prove the following lemma as an intermediate step to Lemma \ref{lem:linkcontracted} in Section \ref{sec:littlebramble}.
\begin{lemma}\label{lem:linkcontracted}
Let $G$ be a $(36k^3 + 2k)$-strongly connected directed graph and $\zB$ be a bramble of depth two and size $>188k^3$ in $G$.
Let $(G, S, T)$ be a $k$-linkage problem instance.
Assume $\zB$ is disjoint from $\{s_i, t_i; 1\leq i \leq k\}$.
Then there exist brambles $\zB_S$ and $\zB_T$ with $\zB_T \subseteq \zB_S \subseteq \zB$ such that $S$ is $(36k^3 + 2k)$-connected to $K_{\zB_S}$ in $G(\zB_S; \zB)$ and $T$ is $3k$-connected to $K_{\zB_T}$ in $G(\zB_T; \zB_S)$. 
Also $|\zB_S| - |\zB_T| < 36k^3$.
Moreover we can find $\zB_S$ and $\zB_T$ in time $O(k^4n^2)$.
\end{lemma}

But first, let's see how Lemma \ref{lem:linkcontracted} implies Lemma \ref{lem:linkup}.
\begin{proof}[Proof of Lemma \ref{lem:linkup}]
Consider the brambles $\zB_S$ and $\zB_T$ given by Lemma \ref{lem:linkcontracted}.
Denote by $W$ the vertices in $G(\zB_T; \zB)$ that belong to exactly one bag in $\zB_T$ and to two bags in $\zB_S$.

\begin{claim*}
{There exist $k$ vertex-disjoint paths $P_1, \dots, P_k$ in $G(\zB_T; \zB)\setminus W$ where $P_i$ links $s_i$ to $v_i$, for some $v_i\in K_{\zB_T}$.}
\end{claim*}

Suppose not; then by Menger's theorem there exists a separation $(A,B)$ of order $<k$ in $G(\zB_T; \zB)\setminus W$ separating $S$ from $K_{\zB_T}$.
But then consider the following separation in $G(\zB_S)$.
Let 
$$A' = (A\cap V(G(\zB_S;\zB))) \cup \{v\in K_{\zB_S}: \im(v) \cap A \neq \emptyset \} \cup (K_{\zB_S}\setminus K_{\zB_T})$$ and
$$B' = (B\cap V(G(\zB_S;\zB))) \cup \{v\in K_{\zB_S}: \im(v) \cap B \neq \emptyset \} \cup (K_{\zB_S}\setminus K_{\zB_T}).$$
Intuitively, $(A',B')$ is the separation $(A,B)$ viewed in the graph $G(\zB_S; \zB)$, plus we add the vertices of $K_{\zB_S}\setminus K_{\zB_T}$ to each side.
It's easy to check that $(A',B')$ is a separation in $G(\zB_S; \zB)$, since every vertex in $V(G(\zB_T; \zB)) \setminus V(G(\zB_S; \zB))$ belongs to $\im(v)$ for some $v\in K_{\zB_S}$.
Also, we have $|A'\cap B'| \leq 2|A\cap B| + 36k^3$ because every vertex belongs to at most two bags of $\zB_S$ and every vertex in $W$ belongs to one bag of $\zB_S\setminus \zB_T$.
But this contradicts Lemma \ref{lem:linkcontracted} and proves the claim.

Choose the paths $P_1,\dots,P_k$ so that they are $\zB_T$-minimal in $G(\zB_T; \zB)$.
Let us now view these as paths in the original graph $G$:
Since $V(G(\zB_T; \zB))\setminus K_{\zB_T}\subseteq V(G)$, each vertex in $P_i$ except $v_i$ is a vertex of $G$, for each $1\leq i\leq k$.
So choose $s_i'\in \im(v_i)$ such that there exists an edge from the second to last vertex of $P_i$ to $s_i'$. 
Then let $P_i^s$ be the path obtained from $P_i$ by replacing $v_i$ with $s_i'$.
Notice that $P_i^s$ is a path in $G$.
The paths $P_1^s,\dots,P_k^s$ are internally disjoint, so they satisfy \Afou.

\begin{claim*}
{There exist vertex-disjoint paths $Q_1, \dots, Q_k$ in $G(\zB_T; \zB_S)\setminus \{v_1, \dots, v_k, s_1',\dots,s_k'\}$ where $Q_i$ links $w_i$ to $t_i$ for some $w_i \in  K_{\zB_T}$.
Moreover, the vertices $v_1,\dots,v_k,w_1,\dots,w_k$ are distinct.}
\end{claim*}

Suppose not; then by Menger's theorem, in the graph $G(\zB_T; \zB_S) \setminus \{v_1, \dots, v_k, s_1',\dots,s_k'\}$ there is a separation $(A,B)$ of order $<k$ properly separating $K_{\zB_T}$ from $T$.
But then $(A\cup \{v_1,\dots, v_k,s_1',\dots,s_k'\}, B\cup \{v_1,\dots,v_k,s_1',\dots,s_k'\})$ has order $<3k$ and properly separates $K_{\zB_T}$ from $T$ in $G(\zB_T;\zB_S)$, contradicting Lemma \ref{lem:linkcontracted}.
This proves the claim.

We may also choose the paths $Q_1,\dots, Q_k$ to be $\zB_T$-minimal in $G(\zB_T; \zB_S)$.
Viewing these paths as paths in $G$ as above (symmetrically), we obtain paths $P_1^t, \dots, P_k^t$, with $P_i^t$ joining $t_i'$ to $t_i$.
These paths satisfy \Afiv.

Let $\zB' = \{\im(v): v\in \{v_1,\dots,v_k,w_1,\dots,w_k\} \}$.
For each $i$, set $B_i^s = \im(v_i)$ and $B_i^t = \im(w_i)$.
We now check that the paths $P_1^s, \dots, P_k^s, P_1^t, \dots, P_k^t$ satisfy the seven assertions in the lemma statement.
\Aone, \Atwo, \Afou and \Afiv have already been established.

To see that \Asix holds, note that each of $P_1,\dots,P_k$ is internally disjoint from $K_{\zB_T}$ in $G(\zB_T; \zB)$.
Similarly, the $Q_1,\dots,Q_k$ paths are internally disjoint from $K_{\zB_T}$ in $G(\zB_T; \zB_S)$.
Moreover, by the definition of $G(\zB_T;\zB)$ and $G(\zB_T;\zB_S)$, every vertex not in $K_{\zB_T}$ in either of those graphs belongs to at most one bag of $\zB_T$ and therefore to at most one bag of $\zB'$.
It follows that for each $i$, each internal vertex of $P_i^s$ and $P_i^t$ belongs to at most one bag of $\zB'$, proving \Asix.

To see \Athr, let $1\leq i,j,\ell \leq k$ be distinct.
Suppose for contradiction that some vertex $v$ belongs to $P_i^s \cap P_j^t \cap (B_{\ell}^s \cup B_{\ell}^t)$.
If $v$ is an internal vertex of either $P_i^s$ or $P_j^t$ then $v$ belongs to only one bag of $\zB'$ by \Asix.
Also, if $v = s_i'$ or $t_j'$ then $v$ belongs to two bags of $\zB'$.
We deduce that $v$ is an internal vertex of both $P_i^s$ and $P_j^t$.
Since we found $P_i$ in the graph $G(\zB_T; \zB) \setminus W$, we know $v\notin W$ so $v$ belongs to one bag in $\zB_T$ and one bag of $\zB_S$.
But we found $Q_j$ in the graph $G(\zB_T; \zB_S)$, so $v$ belongs to one bag of $\zB_T$ and two bags of $\zB_S$.
This is a contradiction, proving \Athr.

Finally, let us check \Asev.
Suppose for contradiction's sake that some vertex $v\in V(G)$ belongs to three paths. 
By \Afou and \Afiv, we must have $v\in P_i^s\cap P_j^s\cap P_{\ell}^t$ or $v\in P_i^t\cap P_j^t\cap P_{\ell}^s$ for some $1\leq i, j, \ell \leq k$. 
If $v\in P_i^s\cap P_j^s\cap P_{\ell}^t$, then by \Afou we may assume without loss of generality that $v = s_i'$.
But the path $Q_{\ell}$ was found in a graph not containing $v_i$ or $s_i'$,
so we must have $s_i' \in B_{\ell}^t \cap B_i^s$.
Since $\zB'$ is depth two, $v\notin B_j^s$ so $v$ is an internal vertex of $P_j^s$, contradicting \Asix.
If $v\in P_i^t\cap P_j^t\cap P_{\ell}^s$, then without loss of generality $v=t_i'\in B_i^t = \im(w_i)$.
By the $\zB_T$-minimality of $P_1,\dots,P_k$, $v$ cannot be an internal vertex of $P_{\ell}$ so we have $v=s_{\ell}'\in B_{\ell}^s$.
Since $v$ belongs to two bags, \Asix implies that $v=t_j'$, a contradiction.

It remains to check that we can indeed find these paths in time $O(k^4n^2)$.
Indeed finding the brambles $\zB_S$ and $\zB_T$ takes time $O(k^4n^2)$ using Lemma \ref{lem:linkcontracted}.
Then, the sets of paths $P_1,\dots,P_k$ and $Q_1,\dots,Q_k$ can be found in time $O(n^2)$ according to Theorem \ref{thm:menger}, and from these we can easily get $P_1^s, \dots, P_k^s, P_1^t, \dots, P_k^t$ in linear time.

\end{proof}

The following lemma shows how to solve any linkage problem half-integrally in a depth two bramble, provided the terminals belong to distinct bags.
\begin{lemma}\label{lem:insideclique}
For all $k\geq 2$, let $G$ be a directed graph and let $S' = (s_1', \dots, s_k')$ and $T' = (t_1', \dots, t_k'))$ be two ordered $k$-tuples of vertices in $G$. Suppose $\zB$ is a bramble of depth two in $G$, and $s_1', \dots, s_k', t_1', \dots, t_k'$ belong to distinct bags $B_1^s, \dots, B_k^s, B_1^t, \dots, B_k^t$, respectively of $\zB$.  Then there exist paths $P_1, \dots, P_k$ such that $P_i$ links $s_i'$ to $t_i'$ and, additionally, every vertex of $G$ is in at most two distinct paths $P_i$.  Finally, it also holds that $P_i \subseteq B_i^s \cup B_i^t$ for each $i$, and we can find the paths $P_1, \dots, P_k$ in time $O(kn^2)$.
\end{lemma}

\begin{proof}%[Proof of Lemma \ref{lem:insideclique}]
For each $i$, we obtain $P_i$ as follows. 
By the definition of a bramble, there exist vertices $v_i\in B_i^s $ and $w_i \in B_i^t$ with either $v_i=w_i$ or $(v_i,w_i)\in E(G)$.
Since $B_i^s$ and $B_i^t$ are both strongly connected, there exist a directed path from $s_i'$ to $v_i$ contained in $B_i^s$ and a directed path from $w_i$ to $t_i'$ contained in $B_i^t$. 
Take $P_i$ to be the concatenation of these two paths.
By construction, each $P_i$ belongs to $B_i^s \cup B_i^t$.
Further, since the bags $B_1^s, \dots, B_k^s, B_1^t, \dots, B_k^t$ are distinct, and every vertex in $G$ belongs to at most two distinct bags, it follows that $P_1, \dots, P_k$ is the desired collection of paths.
Each $P_i$ can be found in time $O(n^2)$, and so the overall running time of $O(kn^2)$ follows.
\end{proof}

We can deduce Theorem \ref{thm:main} from Lemmas \ref{lem:linkup} and \ref{lem:insideclique} as follows.
\begin{proof}[Proof of Theorem \ref{thm:main}]
Let $P_1^s, \dots, P_k^s, P_1^t, \dots, P_k^t$ and $s_1', \dots, s_k', t_1', \dots, t_k'$ and $\zB' = B_1^s, \dots, B_k^s, B_1^t, \dots, B_k^t$ satisfy \Aone - \Asev, as given by Lemma \ref{lem:linkup}.

By \Atwo, $G$, $S' = (s_1', \dots, s_k')$ and  $T' = (t_1', \dots, t_k'))$ satisfying the hypothesis of Lemma \ref{lem:insideclique}.
Let $P_1,\dots,P_k$ be the paths guaranteed by that lemma.

For each $1\leq i\leq k$, let $Q_i = P_i^sP_i P_i^t$ be the concatenation of these three paths.
Clearly, each $Q_i$ is a directed walk linking $s_i$ to $t_i$ and therefore contains a directed path from $s_i$ to $t_i$.
We just need to check that the $k$ paths are half-integral.
Suppose for contradiction's sake that some vertex $v\in Q_i\cap Q_j \cap Q_{\ell}$ for some $1\leq i, j ,\ell \leq k$ all distinct.
By symmetry, we can consider four cases.

{\bf Case 1:}  $v\in P_i\cap P_j$.\\
Then, by Lemma \ref{lem:insideclique}, $v\in (B_i^s\cup B_i^t) \cap (B_j^s\cup B_j^t)$, so $v$ belongs to two bags of $\zB'$.
Then by \Asix $v$ is not an internal vertex of $P_{\ell}^s$ or $P_{\ell}^t$, a contradiction.

{\bf Case 2:} $v\in P_i^s\cap P_j^s \cap P_{\ell}$.\\
By \Afou in Lemma \ref{lem:linkup}, we may assume $v = s_i'$, so $v\in P_i$.
Since $v\in P_{\ell}$, it follows $v\in B_i^s \cap (B_{\ell}^s \cup B_{\ell}^t)$.
By \Asix, $v$ is not an internal vertex of $P_j^s$, so $v \in B_j^s$ as well, a contradiction.

{\bf Case 3:} $v\in P_i^t \cap P_j^t \cap P_{\ell}$.\\
By \Afiv in Lemma \ref{lem:linkup}, we may assume $v = t_i'$, so $v\in P_i$. 
Again, since $v\in P_{\ell}$, it follows $v\in B_i^t \cap (B_{\ell}^s \cup B_{\ell}^t)$.
By \Asix, $v$ is not an internal vertex of $P_j^t$ so $v \in B_j^s$ as well, a contradiction.

{\bf Case 4:} $v\in P_i^s \cap P_j^t \cap P_{\ell}$.\\
By Lemma \ref{lem:insideclique} $P_{\ell} \subseteq (B_{\ell}^s \cap B_{\ell}^t)$, but this contradicts \Athr.
By \Athr $v\notin B_{\ell}^s \cap B_{\ell}^t$ so $v\notin P_{\ell}$, a contradiction.

The running time bound of $O(k^4n^2)$ follows from the bounds given by Lemmas \ref{lem:linkup}, \ref{lem:linkcontracted} and \ref{lem:insideclique}.
\end{proof}

\subsection{Finding a bramble to link to}\label{sec:littlebramble}
In this section we prove Lemma \ref{lem:linkcontracted}.
We will need the following easy combinatorial lemma.

\begin{lemma}\label{lem:subsets}
For $k\geq 1$, let $A$ be a set of $k$ elements and suppose $A_1, \dots, A_{k+1}$ are proper subsets of $A$.
Then there exist $i\neq j$ with $A_i\cup A_j \neq A$.
\end{lemma}

\begin{proof}%[Proof of Lemma \ref{lem:subsets}]
We prove the lemma by induction on $k$. 
The base case $k=1$ is trivial.
Assume that the lemma holds for $k-1$, we will show it must also hold for $k$.
Let $A$ and $A_1, \dots, A_{k+1}$ be as in the lemma statement.
Choose an element $v\in A$ so that (possibly after relabelling) $A_1\setminus\{v\}, \dots, A_{k}\setminus\{v\}$ are proper subsets of $A\setminus v$.
By induction, there exist $i\neq j$ with $(A_i \cup A_j)\setminus\{v\}  \neq A\setminus \{v\}$.
We deduce that $A_i\cup A_j \neq A$.
\end{proof}

We will actually prove Lemma \ref{lem:linkcontracted} one side at a time, applying the following lemma and then a symmetric version of it. 
\begin{lemma}\label{lem:linkoneside}
Let $k,\alpha, \beta$ be integers with $\beta \geq \alpha\geq k$.
Suppose $(G, S = (s_1, \dots, s_k), T = (t_1, \dots, t_k))$ is a $k$-linkage problem instance where $G$ is a $\beta$-strongly connected directed graph 
and $\zB$ be a bramble of depth two and size $> 4k\alpha^2$ in $G$.
Assume $\zB$ is disjoint from $\{s_i,t_i; 1\leq i \leq k\}$.
Then there exists a bramble $\zB_S \subseteq \zB$ such that $S$ is $\alpha$-connected to $K_{\zB_S}$ in $G(\zB_S; \zB)$.
Also $|\zB| - |\zB_S| < 4k\alpha^2$.
Moreover we can find $\zB_S$ in time $O(k\alpha n^2)$.
\end{lemma}

\begin{proof}[Proof of Lemma \ref{lem:linkcontracted}]
Assume $G$ is $(36k^3 + 2k)$-strongly connected and $\zB$ is a bramble of depth two and size $>188k^3$ in $G$.
By Lemma \ref{lem:linkoneside}, there is a bramble $\zB_S\subseteq \zB$ such that $S$ is $(36k^3 + 2k)$-connected to $K_{\zB_S}$ in $G(\zB_S; \zB)$.
We have $|\zB_S| > 188k^4 - (144k^3 + 8k^2) > 36k^3$.
With a proof symmetric to that of Lemma \ref{lem:linkoneside}, one can show that there exists a bramble $\zB_T\subseteq \zB_S$ such that $T$ is $3k$-connected to $K_{\zB_T}$ in $G(\zB_T; \zB_S)$.
Moreover $|\zB_S| - |\zB_T| < 36k^3$.
The running time of $O(k^4n^2)$ follows from the running time bound given in Lemma \ref{lem:linkoneside}.
\end{proof}

\begin{proof}[Proof of Lemma \ref{lem:linkoneside}]

Let $G_0 = G(\zB; \zB)$ and let $K_0 = K_{\zB}$ be the clique of contracted vertices in $G_0$.
% Let $M = 2k^2$.
Set $M = 2k\alpha$.
To find the bramble $\zB_S$ we are looking for, we will generate a sequence of brambles, each contained in the previous one, until we find one that $S$ is sufficiently highly connected to for our purposes.
Consider Algorithm \ref{alg:cutsequence}, a procedure to find the sequence of brambles $\zB \supseteq \zB_1 \supseteq \dots \supseteq \zB_M$ and graphs $G_1, \dots, G_{M} = G(\zB_1;\zB),\dots, G(\zB_M;\zB)$, as well as separations $(A_i',B_i')$ and $(A_i,B_i)$ separating $S$ from the bramble in the graphs $G_{i-1}$ and $G_i$, respectively.

\begin{algorithm}[H]\label{alg:cutsequence}
% \LinesNumbered
\LinesNotNumbered
\DontPrintSemicolon
% \KwIn{A }
% \KwOut{A .}
\For{$i = 1$ \KwTo $M-1$}{
Let $(A_i',B_i')$ be the separation in $G_{i-1}$ that properly separates $S$ from $K_{i-1}$ of minimum order. 

$C_i \longleftarrow K_{i-1} \cap (A_{i}' \cap B_{i}')$.\\ 
$K_{i} \longleftarrow K_{i-1} \setminus C_i$. \\

$\mathcal{C}_{i} \longleftarrow \{\im(v): v\in C_i\}$.\\
$\zB_i \longleftarrow \zB_{i-1} \setminus \zC_{i}$.\\
$G_i \longleftarrow G(\zB_i; \zB)$.\\

$A_{i} \longleftarrow (A_{i}' \setminus C_i )\cup (\cup_{C \in \mathcal{C}_{i} }C)$.\\
$B_{i} \longleftarrow (B_{i}' \setminus C_i) \cup (\cup_{C \in \mathcal{C}_i }C)$.\\
}
\caption{Generating the brambles $\zB_1, \dots, \zB_M$}

\end{algorithm}

We remark that for each $i$ $(A_i', B_i')$ is a separation in the graph $G_{i-1}$ while $(A_{i},B_{i})$ is a separation in $G_{i}$.
Obesrve that $(A_i', B_i')$ properly separates $S$ from $K_{i-1}$ and $(A_i, B_i)$ properly separates $S$ from $K_i$.
Observe also that $A_{i}\cap B_{i}$ is disjoint from $K_{i}$. 

Our main claim is the following.

\begin{claim}\label{clm:klink}
For some $1\leq m \leq M$, the separation $(A_m', B_m')$ has order at least $\alpha$.
\end{claim}

Let's first check that Claim \ref{clm:klink} implies the lemma.
Using Algorithm \ref{alg:cutsequence}, we can find $G_{m-1}$ and the separation $(A_m', B_m')$ given by Claim \ref{clm:klink} in time $O(Mn^2) = O(k\alpha n^2)$.
Take $\zB_S = \zB_{m-1}$.
In each loop of Algorithm \ref{alg:cutsequence}, we remove less than $2\alpha$ bags from the bramble.
It follows that $|\zB_{m-1}| > |\zB|-2M\alpha \geq |\zB|- 4k\alpha^2 $.

To prove Claim \ref{clm:klink} we need the following intermediate claim.
Suppose for the sake of contradiction that $|A_m'\cap B_m'| <\alpha$ for each $1\leq m \leq M$. 

\begin{claim}\label{clm:goodcuts}
We can find indices $1\leq i < j \leq M$ with 
\begin{enumerate}
\item $A_{j}'\cap B_{j}' \cap C = \emptyset$, for each $C\in \mathcal{C}_i$ and with
\item  $(A_i \cap A_j') \setminus (B_i \cap B_j')$ nonempty.
\end{enumerate}
\end{claim}

\begin{proof}%[Proof of Claim \ref{clm:goodcuts}]
We first obtain a set of $k+1$ indices $J\subset [M]$ such that for each $i<j \in J$ we have $A_{j}'\cap B_{j}' \cap C = \emptyset$, for each $C\in \mathcal{C}_i$.
Algorithm \ref{alg:findindices} gives a procedure to find $J$.

\begin{algorithm}[H]\label{alg:findindices}
$J_1 \longleftarrow \{M\}$\\
$M_1 \longleftarrow [M]\setminus (J_1\cup \{i: \exists C\in \mathcal{C}_i \mbox{ with } A_{M}'\cap B_{M}' \cap C \neq \emptyset \})$\\
% Then for each $2\leq i \leq k$, set
\For{i=2 \KwTo k}{ 
$J_i \longleftarrow J_{i-1}\cup \{\max {\ell}: \ell \in M_{i-1} \}$\\
$M_i \longleftarrow M_{i-1}\setminus (J_i\cup \{i: \exists C\in \mathcal{C}_i \mbox{ with } A_{\ell}'\cap B_{\ell}' \cap C \neq \emptyset \mbox{ for some } \ell \in J_i \} )$\\
}
$J = J_{k+1} \longleftarrow J_{k} \cup \{\max {\ell}:  \ell \in M_{i-1} \}$\\
\caption{Generating the set $J$}
\end{algorithm}

Let us first check that $J$ is well-defined.
Observe that $|M_1| > M-2\alpha$, since the sets $\{\mathcal{C}_{\ell}: \ell\in [M]\}$ correspond to distinct sets of bags of $\zB$ and no vertex of $A_M'\cap B_M'$ belongs to more than two bags. (Recall that $|A_M'\cap B_M'|< \alpha$.)
Similarly $|M_i| \geq |M_{i-1}| - 2\alpha$ for each $i\geq 2$.
Since $M = 2k\alpha$, we have $|M_k| > M - 2k\alpha \geq 1$.
Thus $|J| = k+1$.

Now let us check that any pair of indices $i<j \in J$ satisfies the first property in the claim, that is that $A_{j}'\cap B_{j}' \cap C = \emptyset$ for each $C\in \mathcal{C}_i$. 
We will show this holds for $J_1, \dots, J_{k+1}=J$ inductively. 
Since $|J_1|=1$, the base case holds trivially. 
Assuming the property holds for $J_{\ell-1}$, let us show it holds for $J_{\ell}$. Observe that $M_{\ell-1} \subset M_{\ell-2} \subset \dots \subset M$. The unique element $i \in J_{\ell}\setminus J_{\ell-1}$ is the smallest index in $J_{\ell}$, so we just need to show that for each $j\in J_{\ell-1}$ we have $A_{j}'\cap B_{j}' \cap C = \emptyset$ for each $C\in \mathcal{C}_{i}$. But $i\in M_{j-1}$ for each $j$ so this is true by the choice of $i$.

Now, we want two indices $i<j \in J$ satisfying the second property in the claim.
For each $\ell \in J$, recall that $S \subseteq A_{\ell}'$ and $S\subseteq A_{\ell}$ and that $S\cap A_{\ell}'\cap B_{\ell}' = S\cap A_{\ell}\cap B_{\ell} \neq S$.
It follows from Lemma \ref{lem:subsets} that there exist indices $i < j$ with $(S\cap A_i\cap B_i) \cup (S\cap A_j'\cap B_j') \neq S$.
Since $S \subseteq A_i\cap A_j'$, it follows that $(A_i \cap A_j') \setminus (B_i \cap B_j')$ nonempty.

This proves Claim \ref{clm:goodcuts}.

\end{proof}

We can now prove Claim \ref{clm:klink}.

\begin{proof}[Proof of Claim \ref{clm:klink}]
Recall our assumption that $|A_m'\cap B_m'| <\alpha$ for each $1\leq m \leq M$. 
Fix $i$ and $j$ as in Claim \ref{clm:goodcuts}.
The separation $(A_j',B_j')$ is a separation in $G_{j-1}$.
Let $\bar{A_i} = A_i$ and let $\bar{B_i} = B_i\setminus (K_j\setminus K_i) \cup \{C: C\in \mathcal{C}_{\ell}: i< \ell < j \}$.
Observe that $(\bar{A_i}, \bar{B_i})$ is also a separation in $G_{j-1}$ and $(\bar{A_i} \cap \bar{B_i})$ does not meet $K_{j-1}$. 
% In fact $(\bar{A_i} \cap \bar{B_i}) = A_i \cap B_i$.
Further, since each $C\in \mathcal{C}_i$ is strongly connected and meets $B_j'$ and, by the choice of $i$ and $j$, %satisfying Claim \ref{clm:goodcuts},
doesn't meet $A_j'\cap B_j'$, we have that each $C\in \mathcal{C}_i $ belongs to $ \bar{B_i}\setminus A_j'$. 
This implies that 
$$(\bar{A_i}\cap \bar{B_i})\cap A_j' = (A_i\cap B_i)\cap A_j' =  (A_i'\cap B_i')\cap A_j'.$$
Similarly, $$(A_j' \cap B_j') \setminus \bar{B_i} = (A_j' \cap B_j') \setminus B_i = (A_j' \cap B_j') \setminus B_i'.$$

Now, $\bar{A_i}\cap A_j' \subseteq V(G)$, that is, it contains no contracted vertices in $G_{j-1}$.
Thus, while $(\bar{A_i}\cap A_j', \bar{B_i}\cup B_j')$ is a separation in $G_{j-1}$, we may view it as a separation in $G$, namely as $(A_i\cap A_j', V(G)\setminus (A_i\cap A_j' \setminus B_i\cup B_j'))$.
By the second property in Claim \ref{clm:goodcuts}, this separation is nontrivial.
So, by the strong connectivity of $G$ it must have order at least $\beta\geq \alpha$ and we deduce that $|(A_i\cap A_j') \cap (B_i\cup B_j')| \geq \beta$.
Rewriting $|(A_i\cap A_j') \cap (B_i\cup B_j')|$ as $|(A_i\cap B_i) \cap A_j'| + |(A_j' \cap B_j') \setminus B_i|$ we conclude that 
$$|(A_i'\cap B_i')\cap A_j'| + |(A_j' \cap B_j') \setminus B_i'| \geq \beta.$$

We have $|A_i' \cap B_i'| < \alpha$ and $|A_j' \cap B_j'| < \alpha$ by assumption, so we can assume that 
\begin{equation}\label{eqn:modular}
|(A_i'\cap B_i')\setminus A_j'| + |(A_j' \cap B_j') \cap B_i'|  < |A_i' \cap B_i'|.
\end{equation}

Now, let us turn our attention to the graph $G_{i-1}$ and define 
% $$A_j^* = (A_j' \setminus \{C: C\in \mathcal{C}_{\ell}: i< \ell < j \}) \cup \{u: u\in K_{\ell} \mbox{ and } \im_{\ell}(u)\cap A_j'\neq \emptyset \mbox{ for some } \ell \geq i\} $$ and 
% $$B_j^* = (B_j' \setminus \{C: C\in \mathcal{C}_{\ell}: i< \ell < j \}) \cup \{u: u\in K_{\ell} \mbox{ and } \im_{\ell}(u)\cap A_j'\neq \emptyset \mbox{ for some } \ell \geq i\}$$
$$A_j^* = (A_j' \cap V(G_{i-1})) \cup \{u: u\in K_{\ell} \mbox{ and } \im(u)\cap A_j'\neq \emptyset \mbox{ for some } j> \ell \geq i\} $$ and 
$$B_j^* = (B_j' \cap V(G_{i-1})) \cup \{u: u\in K_{\ell} \mbox{ and } \im(u)\cap A_j'\neq \emptyset \mbox{ for some } j> \ell \geq i\} $$
In other words we obtain $A_j^*$ and $B_j^*$ from $A_j'$ and $B_j'$ by replacing any vertices belonging to bags $\zB$ that were `expanded' to obtain the intermediate graphs $G_i,\dots, G_{j-1}$ by the corresponding contracted vertices.
Note that $A_j^* \cap A_i' = A_j' \cap \bar{A_i}$.
We view $(A_j^*,B_j^*)$ as a separation in $G_{i-1}$ and note that %$|A_j^* \cap B_j^*| \leq |A_j'\cap B_j'|$ so 
$$|(A_j^* \cap B_j^*) \cap B_i'| \leq |(A_j' \cap B_j') \cap B_i'|$$ and
$$|(A_i'\cap B_i')\setminus A_j^*| \leq |({A_i}\cap {B_i})\setminus A_j'|.$$

Consider the separation $(A_i'\cup A_j^*, B_i'\cap B_j^*)$ in $G_{i-1}$.
Its order is $|(A_i'\cap B_i')\setminus A_j^*| + |(A_j^* \cap B_j^*) \cap B_i'| \leq |(A_i'\cap B_i')\setminus A_j'| + |(A_j' \cap B_j') \cap B_i'| < |A_i'\cap B_i'|$ by (\ref{eqn:modular}). This contradicts the minimality of $(A_i', B_i')$ in $G_{i-1}$.

This proves Claim \ref{clm:klink}.
\end{proof}
This completes the proof of the theorem.
\end{proof}

%% end katie
%%%%%%%%%%%%

\section{Proofs of Theorems \ref{thm:linked} and \ref{thm:2kconn}}\label{sec:proofs}

Given Theorems \ref{thm:findbramble} and \ref{thm:main}, it is now easy to complete the proofs of Theorems \ref{thm:linked} and \ref{thm:2kconn}.  We begin with Theorem \ref{thm:linked}.

\begin{proof}[Proof of Theorem \ref{thm:linked}]  Let $f_1$ be the function from Theorem \ref{thm:findbramble}.  Let $t = t(k)$ be the value necessary for the size of the bramble in order to apply Theorem \ref{thm:main} and resolve an instance of $\frac12$kDDPP.  

Let $G$ be an $f(t)$-strongly connected graph on $n$ vertices, and let $(G, S = ( s_1, \dots, s_k), T=( t_1, \dots, t_k))$ be an instance of the $\frac12$kDDPP.  We can greedily find a path $P$ with $|V(P)| \ge f(k)$.  Note that any subset of at most $f(k)$ vertices is well-linked, and thus, $V(P)$ is a well-linked set.  By Theorem \ref{thm:findbramble}, we can find in time $O(n^{c_1})$ a bramble $\zB$ of size at least $t$.  As $f(t) \ge 36k^3 + 2k$, by Theorem \ref{thm:main}, we can find a solution to $(G, S, T)$ in time $O(k^4n^2)$, completing the proof of the theorem. 
\end{proof}

For the proof of Theorem \ref{thm:2kconn}, we will need two additional results from \cite{KKfull}.  Note that in \cite{KKfull}, neither statement is algorithmic, but the existence of the algorithm follows immediately from the constructive proof.

\begin{LE}[\cite{KKfull}, 4.3]\label{lem:KK1}
Let $G$ be a directed graph on $n$ vertices and $\zB$ a bramble in $G$.  Then there is a path $P$ intersecting every element of $\zB$ and given $G$ and $\zB$ in input, we can find the path $P$ in time $O(|\zB|n^2)$.
\end{LE}

\begin{LE}[\cite{KKfull}, 4.4]\label{lem:KK2}
Let $G$ be a directed graph graph on $n$ vertices, $\zB$ a bramble of order $k (k + 2)$ and $P$ a path intersecting every element of $\zB$.  Then there exists a set $X \subseteq V(P)$ of order $4k$ which is well-linked. Given $P$, $\zB$, and $G$ in input, we can algorithmically find $X$ in time $|\zB|n^{O(k)}$.  
\end{LE}
  
\begin{proof}[Proof of Theorem \ref{thm:2kconn}]  Let $(G, S = (s_1, \dots, s_k), T = (t_1, \dots, t_k))$ be an instance of the $\frac12$kDDPP.  Let $n = |V(G)|$.  Let $t$ be the necessary size of a bramble in order to apply Theorem \ref{thm:main} to resolve an instance of the $\frac12$kDDPP.  Let $f$ be the function in Theorem \ref{thm:findbramble}.  

By Theorem \ref{thm:brambleortd}, we can either find a tree decomposition of $G$ of width at most $c_2((f(t)+2)^2)$ or a bramble $\zB$ of order $(f(t)+2)^2$.  Given the tree decomposition, by Corollary \ref{cor:solveboundedtw}, we can solve $(G, S, T)$ in time $O(n^{f_1(c_2(f(t)+2)^2)})$ for some function $f_1$.  

If instead with find the bramble $\zB$, in order to apply Theorem \ref{thm:main}, we will have to convert it to a bramble of depth two.  By Theorem \ref{thm:brambleortd}, we may assume that $|\zB| \le n^{2(f(t)+2)^2}$.  Thus, in time $n^{O(f(t)^2)}$, we can find a  path $P$ intersecting every element of $\zB$ by Lemma \ref{lem:KK1}.  By Lemma \ref{lem:KK2}, again in time $n^{O(f(t)^2)}$, we can find a well-linked subset $X \subseteq V(P)$ with $|X| \ge f(t)$.  Finally, applying Theorem \ref{thm:findbramble}, we find can find a bramble $\zB'$ of size $t$ and depth two.  Finally, by Theorem \ref{thm:main}, we can resolve $(G, S, T)$ in time $O(k^4n^2)$.  In total, the algorithm takes time $O(n^{f_2(k)})$ for some function $f_2$, as desired.  
\end{proof}

%!TEX root = journal.tex
\section{Lower Bounds}\label{sec:lowerbounds}
% 
% This section where there still is some work to do.  We would like to show:
% \begin{theorem}
% The $\frac{1}{2}$-kDDPP problem is $W[1]$-hard when parameterized by $k$, even when restricted to problem instances which are strongly $k$-connected.
% \end{theorem}

% Even if we aren't able to show this, we should be able to at least show:

% \begin{theorem}
% The $\frac{1}{2}$-kDDPP problem is NP-complete when $k$ is also given as part of the input, even when restricted to problem instances which are strongly $k$-connected.
% \end{theorem}

%\begin{theorem}\label{thm:npc}
%For each $\epsilon<1$, the $\frac{1}{2}$-kDDPP problem is NP-complete when $k$ is also given as part of the input, even when restricted to problem instances which are strongly $\epsilon k$-connected.
%\end{theorem}
In this section, we give the proof of Theorem \ref{thm:hardness}. 

\begin{proof}[Proof of Theorem \ref{thm:hardness}]%[Proof of Theorem \ref{thm:npc}]
We give a reduction from the satisfiablity problem for Boolean formulas in $3$-CNF (3SAT).
The reduction is a variant on the one given in \cite{FHW} which shows NP-completeness of the integral $2$DDPP in directed acyclic graphs.

Fix a formula $F$ with variables $x_1,\dots,x_n$ and clauses $c_1,\dots,c_m$.
% For each variable $x_i$
Set $k' = 3 + m = n$, choose $M \in \{ \lceil\tfrac{\epsilon}{1-\epsilon}k' \rceil, \lceil\tfrac{\epsilon}{1-\epsilon}k' \rceil + 1 \}$ so that $M$ is even, and set $k = k' + M$. 
Noting that $M \geq \epsilon k$, we will construct a $k$-linkage problem instance with strong connectivity $M$.
We construct the graph $G(F) = (V,E)$ as follows (see Figure \ref{fig:reduction} for an example).
Let 
$$V_1 = \{v_{i,j} : x_i\in c_j\} \cup \{\obar{v_{i,j}}: \obar{x_i}\in c_j\} \cup \{c_j, c_j': 1\leq j \leq m\} \cup \{s_1,s_2,s_3,t_1,t_2,t_3\} \cup \{u_i,v_i, \obar{u_i}, \obar{v_i}\} \cup \{b_i, b_i' ; 1 \leq i \leq n\} $$
and $W = \{w_1,\dots,w_M\}$.
Let $V = V_1 \cup W$.

For each variable $i$, create a directed path $P_i$ from $u_i$ to $v_i$ containing internally each vertex of the form $v_{i,j}$, and a directed path $\obar{P_i}$ from $\obar{u_i}$ to $\obar{v_i}$ containing each vertex of the form $\obar{v_{i,j}}$.

\begin{figure}[t]
\begin{center}
\scalebox{.5}{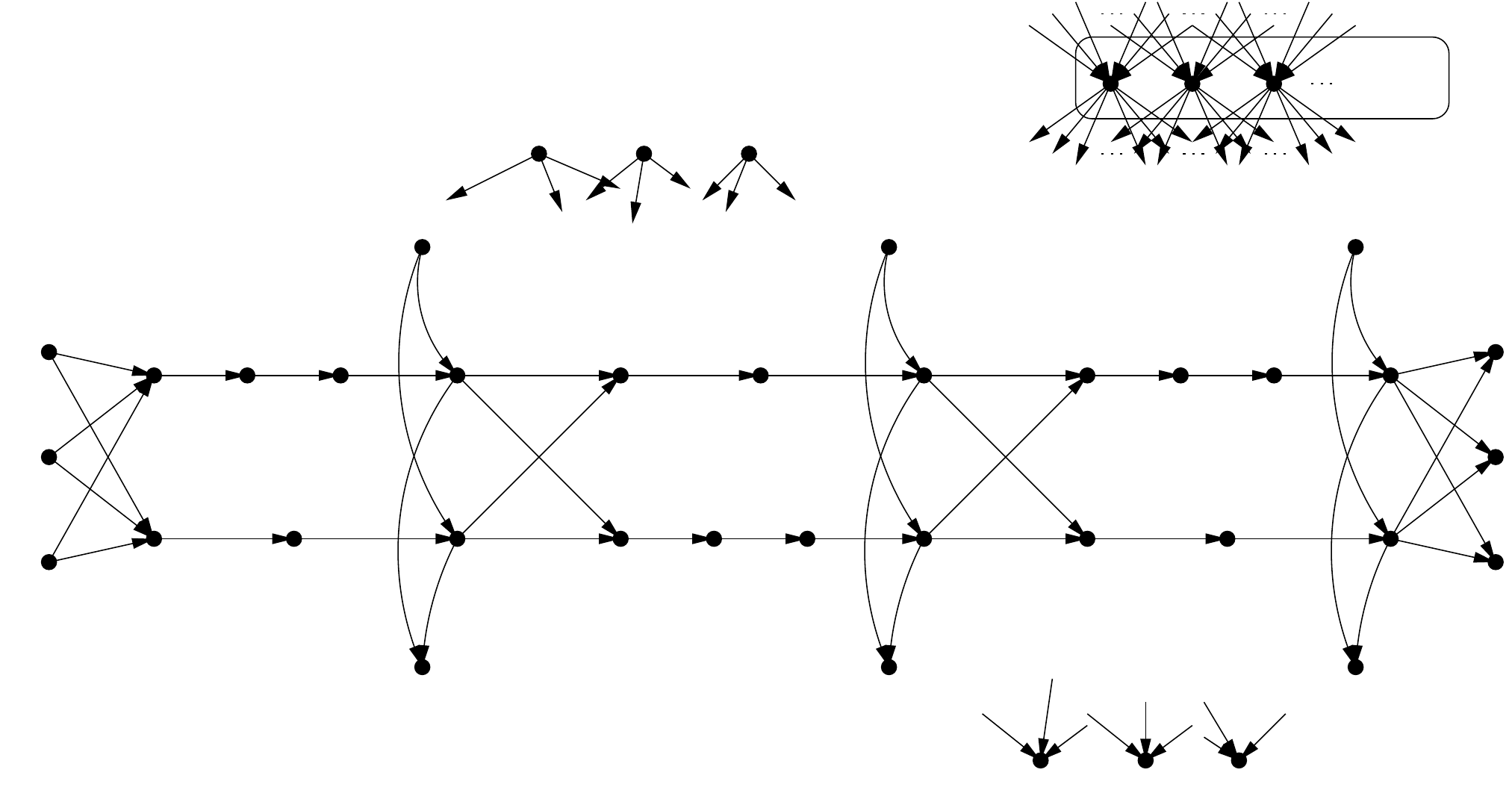}
\end{center}
\caption{Illustration of the graph $G(F)$ for the formula 
$F = (x_1\lor x_2\lor x_3)\land (x_1\lor \obar{x_2}\lor x_3)\land (\obar{x_1}\lor \obar{x_2}\lor \obar{x_3}) $}
\label{fig:reduction}
\end{figure}

Let 
\begin{align*}
E_1 &= \{E(P_i) \cup E(\obar{P_i}) : 1\leq i \leq n\}; \\
E_2 &= \{(v_i, u_{i+1}), (v_i, \obar{u_{i+1}}), \obar{v_i}, u_{i+1}), (\obar{v_i}, \obar{u_{i+1}}): 1\leq i < n \}; \\
E_3 &= \{(c_j, v_{i,j}), (v_{i,j}, c_j') : v_{i,j}\in V_1\} \cup \{(c_j, \obar{v_{i,j}}), (\obar{v_{i,j}}, c_j'): \obar{v_{i,j}}\in V_1 \};\\
E_4 &= \{(b_i, v_i), (b_i, \obar{v_i}), (v_i, b_i), (\obar{v_i}, b_i)\};\\
E_5 &= \{ (s_i, u_1), (s_i, \obar{u_1}) : i\in \{1,2,3\} \} \cup \{ (v_n, t_i), (\obar{v_n}, t_i) : i\in \{1,2,3\} \};\\
E_6 &= \{ \{(w_i, v), (v, w_i): v\in V_1 \cup \{w_j: j<i\} \} : 1\leq i \leq M\}.
\end{align*}
% $$E_1 = \{E(P_i) \cup E(\obar{P_i}) : 1\leq i \leq n\}, $$
% $$E_2 = \{(v_i, u_{i+1}), (v_i, \obar{u_{i+1}}), \obar{v_i}, u_{i+1}), (\obar{v_i}, \obar{u_{i+1}}): 1\leq i < n \},$$
% $$E_3 = \{(c_j, v_{i,j}), (v_{i,j}, c_j') : v_{i,j}\in V_1\} \cup \{(c_j, \obar{v_{i,j}}), (\obar{v_{i,j}}, c_j'): \obar{v_{i,j}}\in V_1 \},$$
% $$E_4 = (b_i, v_i), (b_i, \obar{v_i}), (v_i, b_i), (\obar{v_i}, b_i),$$
% $$E_5 = \{ (s_i, u_1), (s_i, \obar{u_1}) : i\in \{1,2,3\} \} \cup \{ (v_n, t_i), (\obar{v_n}, t_i) : i\in \{1,2,3\} \},$$
% $$E_6 = \{ \{(w_i, v), (v, w_i): v\in V_1 \cup \{w_j: j<i\} \} : 1\leq i \leq M\}.$$

Set $E = E_1\cup E_2 \cup E_3 \cup E_4 \cup E_5 \cup E_6$.
Observe that $G(F)$ has strong connectivity at least $M$, since each $w_i$ is adjacent to and from each other vertex.

Consider the $k$-linkage problem instance $(G(F),S,T)$ where 
$$S = \{s_1,s_2,s_3, c_1,\dots,c_m, b_1, \dots, b_n, w_1, \dots, w_{\tfrac M2}, w_{\tfrac M2 + 1}, \dots, w_M \}$$
and $$T = \{t_1,t_2,t_3, c_1',\dots,c_m', b_1',\dots, b_n', w_{\tfrac M2 + 1}, \dots, w_M, w_1, \dots, w_{\tfrac M2} \}.$$

This is clearly a polynomial reduction.
We need to show that $F$ is satisfiable if and only if $(G(F),S,T)$ is half-integrally feasible.

\begin{claim*}
If $F$ is satisfiable then $(G(F),S,T)$ is half-integrally feasible.
\end{claim*}
Fix a satisfying assignment of $F$ and consider the $S$-$T$ linkage consisting of 
\begin{itemize}
	\item the paths $s_1P_1\dots P_n t_1$, $s_2\obar{P_1}\dots \obar{P_n} t_2$;
	\item the $s_3-t_3$ path whose interior is the concatenation of $P_i$ if $x_i$ is set to false or $\obar{P_i}$ if $x_i$ is set to true, for each $1\leq i \leq n$;	
	\item for each $1\leq j \leq m$, a path $c_j v_{i,j} c_j'$ if $c_j$ contains a variable $x_i$ that is true, or the path $c_j \obar{v_{i,j}} c_j'$ if $c_j$ contains the negation of a variable $x_i$ that is false;
	\item for each $1\leq i\leq n$, the path $b_i v_i b_i'$ if $x_i$ is true, or the path $b_i \obar{v_i} b_i'$ if $x_i$ is false;
	\item the 1-edge paths $(w_i,w_{\tfrac M2 + i})$ and $(w_{\tfrac M2 + i}, w_i)$ for each $1\leq i \leq \tfrac M2$.
\end{itemize}

It is straightforward to check that each vertex is used at most twice in this linkage, proving that $(G(F),S,T)$ is half-integrally feasible.

\begin{claim*}
If $(G(F),S,T)$ is half-integrally feasible then $F$ is satisfiable.
\end{claim*}
Suppose $(G(F),S,T)$ has a half-integral solution.
Observe that each vertex in $W$ belongs to both $S$ and $T$ and therefore cannot belong to any path in the solution for which it is not a terminal.
Thus, for $j=1,2,3$, every $s_j-t_j$ path must use either $P_i$ or $\obar{P_i}$, for each $1\leq i \leq n$.
We will show that the following assignment satisfies $F$: set variable $x_i$ to false if the path $\obar{P_i}$ is used by two of those paths, and to true if the path $P_i$ is used twice by those paths.

Every $b_i-b_i'$ path must use $v_i$ or $\obar{v_i}$. 
By half-integrality, the $b_i-b_i'$ path uses $v_i$ if $\obar{P_i}$ is used by two of the $s_j-t_j$ paths, and uses $\bar{v_i}$ if $P_i$ is used by two of those paths.
In particular, each $v_i$ and $\bar{v_i}$ must be used exactly twice by the union of the $s_j-t_j$ and $b_i-b_i'$ paths.
We deduce that each $c_j-c_j'$ path contains two edges.
Moreover, the middle vertex of that path is either $v_{i,j}$ where the variable $x_i$ belongs to $c_i$ and is set to true, 
or $\obar{v_{i,j}}$ where the negation of variable $x_i$ belongs to $c_i$ and $x_i$ is set to false.
This proves that $F$ is satisfiable.

\end{proof}
%%%%%%%%%%%%%%%%%%%%%%%%%%%%%%%%%%
%%%%%%%%%%%%%%%%%%%%%%%%%%%%%%%%%%

\bibliography{linkage}{}
\bibliographystyle{plain}

\end{document}